\documentclass{amsart}
\usepackage{amssymb, amsmath}
\usepackage[all]{xy}
\usepackage{mathrsfs}
\usepackage{amscd,amssymb,color, amsmath}
\usepackage{hyperref}
\usepackage{todonotes}

\newcommand{\norm}{N}
\newcommand{\normR}{{}_R\norm}

\newcommand{\Tr}{tr}

\newcommand{\m}[1]{{\protect\underline{#1}}}

\newcommand{\mM}{\m{M}}

\newcommand{\mSet}{\m{\Set}}
\newcommand{\mcC}{\m{\cC}}

\newcommand{\cc}[1]{\mathcal #1}
\newcommand{\cC}{\cc{C}}

\newcommand{\cO}{\cc{O}}
\newcommand{\ccZ}{\cc{Z}}

\newcommand{\cL}{\cc{L}}

\newcommand{\Sp}{\mathcal Sp}

\newcommand{\Set}{\mathcal Set}

\newcommand{\Comm}{\mathcal Comm}
\newcommand{\Ninfty}{N_\infty}

\newcommand{\Ab}{\mathcal Ab}

\newcommand{\cOrb}{\mathcal Orb}

\newcommand{\Mackey}{\mathcal Mackey}

\newcommand{\Symm}{\mathcal Sym}

\newcommand{\GComm}{G\mhyphen\Comm}

\newcommand{\Coeff}{\mathcal Coeff}
\newcommand{\mCoeff}{\m{\Coeff}}
\newcommand{\Mod}{\mathcal Mod}
\newcommand{\RMod}{R\mhyphen\Mod}
\newcommand{\mRMod}{R\mhyphen\m{\Mod}}

\newcommand{\To}{\Rightarrow}

\mathchardef\mhyphen=45

\numberwithin{equation}{section}

\newtheorem{theorem}{Theorem}[section]

\newtheorem{corollary}[theorem]{Corollary}

\newtheorem{proposition}[theorem]{Proposition}

\theoremstyle{remark}
\newtheorem{remark}[theorem]{Remark}
\newtheorem{example}[theorem]{Example}

\newtheorem{warning}[theorem]{Warning}

\theoremstyle{definition}

\newtheorem{definition}[theorem]{Definition}

\newcommand{\defemph}[1]{\textbf{#1}}

\begin{document}

\title{Equivariant symmetric monoidal structures}

\author[M.A.~Hill]{Michael A.~Hill}
\address{University of California Los Angeles \\ Los Angeles, CA 90095}
\email{mikehill@math.ucla.edu}
\thanks{M.~A.~Hill was supported in part by NSF DMS--1207774 and the Sloan Foundation} 

\author[M.J.~Hopkins]{Michael J.~Hopkins}
\address{Harvard University \\ Cambridge, MA 02138}
\email{mjh@math.harvard.edu}
\thanks{M.~J.~Hopkins was supported in part by NSF DMS--0906194}
\thanks{Both authors were supported by DARPA through the Air Force Office of Scientific Research (AFOSR) grant number HR0011-10-1-0054}

\begin{abstract}
Building on structure observed in equivariant homotopy theory, we define an equivariant generalization of a symmetric monoidal category: a $G$-symmetric monoidal category. These record not only the symmetric monoidal products but also symmetric monoidal powers indexed by arbitrary finite $G$-sets. We then define $G$-commutative monoids to be the natural extension of ordinary commutative monoids to this new context. Using this machinery, we then describe when Bousfield localization in equivariant spectra preserves certain operadic algebra structures, and we explore the consequences of our definitions for categories of modules over a $G$-commutative monoid.
\end{abstract}



\maketitle

\section{Equivariant Algebra and Homotopy}

One of the primary stumbling blocks in equivariant homotopy theory is the transfer. This has many guises in the literature, and one often sees words like ``naive'' and ``genuine'' used to describe indexing categories, the invertibility of representation spheres, or the existence of transfers. For finite groups, however, all of these notions are equivalent \cite{BlumbergCont}.

The approach we take is to focus on the transfers, seeing how this structure plays out more generally. In particular, we describe a general structure on a symmetric monoidal category (really a sequence of categories indexed by conjugacy classes of subgroups) that mirrors the fixed points of the action of $G$ on a category. Depending on which transfers exist, the structure we see will interpolate between a ``naive'' version (having no transfers) and a ``genuine'' version (having all transfers). All of our discussion takes place at the level of the symmetric monoidal categories. If our categories have extra structure (like, say, a model structure), then we can ask for compatibility of our structure with the model structure (equivalently, asking if our functors appropriately derive). We will not address that point in this note, as it can be quite subtle. For $G$-spectra, this was the entirety of Appendix B in Hill-Hopkins-Ravenel \cite{HHR}.

The results of this paper were announced at the Stanford Conference ``Algebraic topology: applications and new directions'', and a related write-up was included in the conference proceedings \cite{HHInversion}.

\subsection*{Categorical Assumptions}
In all that follows, we will use the word ``functor'' to mean ``pseudofunctor''. In particular, commutativity of diagrams is only guaranteed up to (unspecified) isomorphism.

\subsection*{Acknowledgements}
We have profited from discussions with many people during the writing of this, and many people gave important and constructive feedback. We single out (alphabetically) Clark Barwick, Andrew Blumberg, Anna Marie Bohmann, Tyler Lawson, Peter May, Angelica Osorno, and Emily Riehl. 

\section{Symmetric monoidal coefficient systems and symmetric monoidal Mackey functors}
A central tenet in equivariant algebraic topology is that one never considers a group $G$ by itself; one always considers $G$ together with all of its subgroups. In this foray into equivariant category theory, we make use of this maxim, and we begin with structures analogous to the classical algebraic notions of a coefficient system and a Mackey functor.

\begin{definition}
Let $\cOrb_{G}$ denote the orbit category of $G$, and let $\Symm$ be the category of symmetric monoidal categories and strong monoidal functors.

A {\defemph{symmetric monoidal coefficient system}} is a contravariant functor $\m{\cC}\colon \cOrb_{G}\to\Symm$. The value of the functor on a map is called a ``restriction'' map.

A {\defemph{symmetric monoidal Mackey functor}} is a pair of functors $\mM$, one covariant and one contravariant, from $\cOrb_{G}\to \Symm$ which agree on objects and for which we have the standard double coset formula (up to isomorphism). The contravariant maps are restrictions, and the covariant maps are transfers.
\end{definition}

\begin{remark}
Since the first drafts of this paper, two very natural higher categorical enrichments of the definition of a symmetric monoidal Mackey functor have been worked out:
\begin{enumerate}
\item Bohmann-Osorono consider a $2$-categorical version permutative Mackey functors where their $2$-cells are exactly enumerating coherency conditions for the double coset isomorphisms above \cite{BohOso}.
\item Barwick-Dotto-Glasman-Nardin-Shah have produced the natural $\infty$-categorical extension as a kind of Cartesian fibration over the category of finite $G$-sets \cite{Bourbon1, Bourbon2}, analogous to Lurie's reformulation of a symmetric monoidal product \cite{LurieAlgII}.
\end{enumerate}
\end{remark}

In many of the natural examples, the restriction maps are actually strict symmetric monoidal functors, while the transfer maps are only strong monoidal.

In particular, in a symmetric monoidal Mackey functor, we have for each subgroup $H\subset G$ a symmetric monoidal category $\mM(G/H)$, and we have symmetric monoidal restriction and transfer maps linking these symmetric monoidal categories.

Just as the simplest example of a symmetric monoidal category is a commutative monoid, the simplest example of a symmetric monoidal Mackey functor is a Mackey functor.

\begin{example}
If $\mM$ is a Mackey functor, then $\mM$ naturally defines a permutative monoidal Mackey functor. The value on $G/H$ is just $\mM(G/H)$ viewed as a discrete category with symmetric monoidal product given by the addition on $\mM(G/H)$. The restriction and transfer maps are the ones coming from $\mM$.
\end{example}

The next example is the most important one, playing the role in this theory that the category of finite sets does for ordinary symmetric monoidal categories.

\begin{example}
Let $\mSet$ be the functor which associates to $G/H$ the category of finite $G$-sets over $G/H$. The disjoint union of finite $G$-sets makes this into a symmetric monoidal category, and the pullback along a map $G/K\to G/H$ defines a strong symmetric monoidal restriction functor
\[
\mSet(G/H)\to\mSet(G/K).
\]
Composition with the map $G/K\to G/H$ gives the transfer map
\[
\mSet(G/K)\to\mSet(G/H),
\]
and the compatibility is the usual Mackey double-coset formula.
\end{example}

\begin{remark}
The assignment
\[
(T\xrightarrow{p} G/H)\mapsto p^{-1}(eH)
\]
gives an equivalence of categories between $\mSet(G/H)$ and $\Set^H$. In all that follows, we will blur this distinction, referring to objects in $\mSet(G/H)$ as $H$-sets. The functoriality of the construction is most easily understood using this comma category definition, however.
\end{remark}

Finite $H$-sets are of course $H$-objects in the category of finite sets. The symmetric monoidal Mackey functor structure on $\mSet$ is actually generic, applying to $H$-objects in a symmetric monoidal category. Again, for transparency, we choose the equivalent model
\[
\mcC^H\equiv \mcC^{B_{G/H}},
\]
where $B_{G/H}$ is the translation category of the $G$-set $G/H$. The associated $G$-symmetric monoidal structure is studied extensively in Hill-Hopkins-Ravenel.

\begin{example}[{\cite[A.3]{HHR}}]\label{ex:GObjects}
If $\cC$ is a symmetric monoidal category, then the category of objects of $\cC$ with a  $G$-action, $\cC^{G}$ naturally extends to a symmetric monoidal Mackey functor. The underlying functor on objects for $\mcC$ is
\[
G/H\mapsto\cC^{B_{G/H}},
\]
and the restriction maps arise from the natural maps on the indexing categories. 

For the transfer maps, we observe that these quotient maps, when translated into translation categories become finite covering categories \cite[A.3]{HHR}. Associated to such a covering category, we have canonical strong symmetric monoidal transfer maps. 
\end{example}

Example~\ref{ex:GObjects} applies in a huge number of contexts: in $G$-spectra, it produces the Hill-Hopkins-Ravenel norm; in $G$-modules, it provides essentially the Evens norm \cite{Evens}; and in $G$-spaces with the cartesian product or disjoint union, it gives coinduction or induction respectively.

Combining the transfers and restrictions allows us to take symmetric monoidal products of any diagram indexed by finite $G$-sets, rather than just by finite sets. This provides a convenient language for understanding the structure present on mathematical objects like the categories of acyclics for a localization or for categories of modules over a ring.

\section{\texorpdfstring{$\mSet$}{Set} and \texorpdfstring{$\mSet$}{Set}-modules}
%
We begin with an observation. For any symmetric monoidal category $\cC$ and for any finite set $T$ (here we have no equivariance), we have a strong symmetric monoidal functor $\cC\to\cC$ given by
\[
M\mapsto \bigotimes_{T}M=: T \Box M.
\]
This is natural in isomorphisms in the $T $ factor and in any map in the $M$ factor, and thus gives a strong bilinear functor
\[
-\Box-\colon \Set^{Iso}\times\cC\to\cC.
\]
We call this the canonical exponentiation map. Our notion of genuine $G$-symmetric monoidal structures and the more general $\cO$-symmetric monoidal structures are simply extensions of this functor over $\mSet$ using the canonical fully faithful inclusion of $\Set$ (viewed as a constant coefficient system) into $\mSet$ as the trivial objects.

\subsection{Genuine $G$-Symmetric Monoidal Structures}
The category of $G$-sets has a second symmetric monoidal structure: cartesian product. This is compatible with the restriction map in the sense that the restriction functor is also a strong symmetric monoidal functor for this symmetric monoidal product. We can package this data in the language of multicategories, building on work of Elmendorf-Mandell and Bohmann-Osorno connecting symmetric monoidal categories and multicategories.

There is a multicategory of symmetric monoidal categories, where the multimaps from a collection of symmetric monoidal categories $\cC_1, \dots, \cC_n$ to a symmetric monoidal category $\mathcal D$ are functors of underlying categories
\[
\cC_1\times\dots\times\cC_n\to\mathcal D
\]
that are strong symmetric monoidal in each factor separately (see \cite[Def 3.2]{EM06} for more detail). By analogy with vector spaces, a map in the multicategory of symmetric monoidal categories is called ``multilinear''. This extends easily to a multicategory of symmetric monoidal coefficient systems using  a trivial observation.

\begin{proposition}
If $\mcC_1$ and $\mcC_2$ are symmetric monoidal coefficient systems, then the assignment
\[
G/H\mapsto\mcC_1(G/H)\times\mcC_2(G/H)
\]
is a category-valued coefficient system. This is the product in symmetric monoidal coefficient systems.
\end{proposition}

In this way, we can make sense of a multicategory of symmetric monoidal coefficient systems. The objects are symmetric monoidal coefficient systems, and the multi-maps are natural transformations which for each $G/H\in\cOrb$ are multilinear maps.

The coefficient system $\mSet$ will play an important role here, since it becomes a kind of commutative monoid in coefficient systems.

\begin{proposition}
The Cartesian product induces a bilinear map
\[
\mSet\times\mSet\to\mSet
\]
which takes a pair of finite $H$-sets to their Cartesian product (equivalently the fiber product in the slice category over $G/H$).
\end{proposition}

Our notion of a [genuine] $G$-symmetric monoidal structure on a symmetric monoidal coefficient system $\mcC$ is then the notion of a module over $\mSet$.

\begin{definition}
Let $\mcC$ be a symmetric monoidal coefficient system functor. Then a [genuine] $G$-symmetric monoidal structure on $\mcC$ is a bilinear map 
\[
-\Box-\colon\mSet^{Iso}\times\mcC\to\mcC,
\]
making $\mcC$ into a module over $\mSet^{Iso}$ in the sense that

\begin{enumerate}
\item for each $H\subset G$, when restricted to $\Set^{Iso}\subset \mSet(G/H)$, this is the canonical exponential map on the symmetric monoidal category $\mcC(G/H)$, and
\item the following diagram of symmetric monoidal coefficient systems commutes up to natural isomorphism:
\[
\xymatrix{
{\mSet^{Iso}\times\mSet^{Iso}\times\mcC}\ar[r]^-{1\times \Box}\ar[d]_{(-\times-) \times 1} & {\mSet^{Iso}\times\mcC}\ar[d]^{\Box} \\
{\mSet^{Iso}\times\mcC}\ar[r]_-{\Box} & {\mcC.}}
\]
\end{enumerate}
\end{definition}

Even though the extension of the canonical exponential map is part of the data of a $G$-symmetric monoidal category structure on a symmetric monoidal coefficient system $\mcC$, we will follow the standard conventions in category theory and refer only to $\mcC$.

\begin{definition}
If $\mcC_1$ and $\mcC_2$ are $G$-symmetric monoidal categories, then a strong $G$-monoidal functor $F\colon\mcC_1\to\mcC_2$ is a functor $F$ together with natural isomorphisms
\[
T\Box F(-)\To F(T\Box -).
\]
for all $G/H\in\cOrb_G$ and for all $T\in\mSet(G/H)$.

Lax, weak, and the co-versions are defined analogously.
\end{definition}

If $\mcC$ is actually a symmetric monoidal Mackey functor, then there is a canonical $G$-symmetric monoidal structure, determined by decomposing any $H$-set into orbits.
\begin{proposition}
If $\mcC$ is a symmetric monoidal Mackey functor with transfer maps $Tr_{H}^{K}$ for $H\subset K$, then the assignment
\[
K/H\Box M :=Tr_{H}^{K}i_{H}^{\ast}M
\]
gives $\mcC$ a genuine $G$-symmetric monoidal structure.
\end{proposition}
This is simply a categorical version of Frobenius reciprocity. Classically, Frobenius reciprocity is the statement that the cartesian product and induction are linked via a natural isomorphism:
\[
(G\times_{H} S)\times T \cong G\times_{H}(S\times i_{H}^{\ast} T ).
\]
The $G$-symmetric monoidal structure on a symmetric monoidal Mackey functor is in turn defined by this property. This is a categorification of the classical result that a Mackey functor is canonically a module over the Burnside ring Mackey functor $\underline{A}$. This approach has been studied in beautiful work of Berman which shows that there is a natural, $\infty$-categorial collection of properties which characterize modules over $\mSet$ and the correspondence version thereof \cite{Berman}.

Thus all of the situations considered in the previous section naturally give genuine $G$-symmetric monoidal structures.  In fact, most of the interesting and natural examples known to the authors arise in this way. There is a surprising pathological example, however.

\begin{example}
For any symmetric monoidal category, there is a trivial genuine $G$-symmetric monoidal structure given by simply composing with the strong symmetric monoidal restriction functor $\mSet\to\mSet^{tr}$, where $\mSet^{tr}$ is the constant coefficient system with value $\Set$.

Except in very trivial situations, this does not arise from a symmetric monoidal Mackey functor structure, and nor do any of the genuine $G$-symmetric monoidal structures interpolating between this and a more natural one.
\end{example}

\subsection{Comparisions to symmetric monoidal $G$-categories}

Guillou-May and their coauthors have created an extensive body of work around $G$-categories and the notion of a symmetric monoidal structure therein \cite{GuillouMay}. These are very natural examples of our $G$-symmetric monoidal categories, so we briefly describe how any Guillou-May symmetric monoidal $G$-category gives rise to a symmetric monoidal Mackey functor. This kind of approach is used heavily by Bohmann-Osorono \cite{BohOso} in their construction of Eilenberg-MacLane spectra associated to a Mackey functor. 

Guillou-May consider $G$-categories, by which they mean internal categories in $G$-sets. They then consider the categorical Barrat-Eccles operad $\mathcal E_\infty$ which in level $n$ is simply $\Sigma_n$ viewed as an indiscrete category (so any two objects are uniquely isomorphic). The category of categories is symmetric monoidal under Cartesian product, so Example~\ref{ex:GObjects} fits the category of $G$-categories into a symmetric monoidal Mackey functor. In particular, we can apply the norm functor there to the trivial operad $\mathcal E_\infty$, getting the equivariant Barrat-Eccles operad $\mathcal E_\infty^G$.

Using this operad, Guillou and May define ``symmetric monoidal $G$-categories'' as the pseudo-algebras for the categorical operad $\mathcal E_\infty^G$ (the actual algebras are permutative $G$-categories). The operad $\mathcal E_\infty^G$ parameterizes all possible norms from finite $H$-sets, and thus any symmetric monoidal $G$-category gives rise to a canonical symmetric monoidal Mackey functor and hence a genuine $G$-symmetric monoidal category. This construction is visibly a functor from the category of symmetric monoidal $G$-categories and (strong, lax, etc)-maps to the category of $G$-symmetric monoidal categories. However, as the example of coefficient systems in Section~\ref{sec:Coeff} shows, not all genuine $G$-symmetric monoidal categories arise in this way.

\subsection{Genuine $G$-Commutative Monoids}
Classically, in a symmetric monoidal category $\cC$, we can describe the commutative monoids using the tensoring operation over $\Set$. An object $M\in\cC$ defines a functor
\[
\Set^{Iso}\to\cC
\]
by $T \mapsto T \Box M$.

\begin{proposition}
A commutative monoid structure on $M$ is equivalent to an extension
\[
\xymatrix{
{\Set^{Iso}}\ar[r]^{-\Box M}\ar[d] & {\cC.} \\
{\Set}\ar@{-->}[ur] & {}
}
\]
\end{proposition}
\begin{proof}
Consider the canonical map in $\Set$ $\nabla\colon\ast\amalg\ast\to\ast$. If we have an extension of $-\Box M$ over all of $\Set$, then we have a multiplication map
\[
M\otimes M\cong (\ast\amalg\ast)\Box M\to \ast\Box M
\cong M.
\]
If $\tau$ is the twist map $\ast\amalg\ast\to\ast\amalg\ast$, then $\nabla\circ\tau=\nabla$. This means that the multiplication on $M$ is commutative. Similarly,
\[
\nabla\circ(\nabla\amalg 1)=\nabla\circ(1\amalg \nabla),
\]
which shows the multiplication is associative. Finally, the map $\emptyset\to\ast$ gives the unit map.

The converse is also standard.
\end{proof}

Given a genuine $G$-symmetric monoidal category $\mcC$, we can make an analogous definition. If $H\subset G$, then let $i_{H}^{\ast}$ denote the symmetric monoidal functor $\mcC({G/G})\to \mcC({G/H})$. Since $G/G$ is a terminal object in the category $\cOrb_{G}$, an object $M$ in $\mcC(G/G)$ canonically defines elements $i_{H}^{\ast}M$ for all other orbits $G/H$. In the following definition, we will also use $M$ for the images in $\mcC(G/H)$ for any $H$.

\begin{definition}
A [genuine] $G$-commutative monoid is an object $M\in\mcC(G/G)$ together with an extension of functors of symmetric monoidal coefficient systems
\[
\xymatrix{
{\mSet^{Iso}}\ar[r]^{-\Box M}\ar[d] & {\mcC.} \\
{\mSet}\ar@{-->}[ur]_{N^{(-)}(M)} & {}
}
\]

For a subgroup $H\subset G$, the map $G/H\Box M\to M$ is the ``norm'' from $H$ to $G$, and is denoted $N_H^G$.

If $\mcC$ is a symmetric monoidal Mackey functor, then we impose a compatibility condition that the following diagram commutes:
\[
\xymatrix{
{N_{H}^{G}(H/K\Box i_{H}^{\ast}M)} \ar[r]^-{N_{K}^{H}} \ar[d]_{\cong} & {N_{H}^{G} i_{H}^{\ast}M\cong G/H\Box M} \ar[r]^-{N_{H}^{G}} & {M.} \\
{G/K\Box M}\ar[urr]_-{N_{K}^{G}}
}
\]
\end{definition}

Just as for categories, although the multiplication and norms on $M$ are an essential piece of the data, we will refer to a $G$-commutative monoid simply as the underlying object.

\begin{definition}
If $M$ and $M'$ are $G$-commutative monoids in $\mcC$, then a map of $G$-commutative monoids from $M$ to $M'$ is a morphism $f\colon M\to M'\in\mcC(G/G)$ such that for all $G/H\in\cOrb_G$ and for all $T\in\mSet(G/H)$, we have a commutative diagram
\[
\xymatrix{
{T\Box M}\ar[r]^{N^T}\ar[d]_{T\Box f} & {M}\ar[d]^{f} \\
{T\Box M'}\ar[r]_{N^T} & {M'.}
}
\]
\end{definition}

Equivalently, a map of $G$-commutative monoids is simply a natural transformations between the two functors $N^{(-)}M\colon \mSet\to\mcC$ and $N^{(-)}M'\colon\mSet\to\mcC$.

Thus a $G$-commutative monoid is a commutative monoid $M$ in $\mcC(G/G)$, together with commutative monoid maps
\[
N_{H}^{G}\colon G/H\Box M\to \ast\Box M\cong M
\]
which are natural in the subgroup $H$. Similarly, a map of $G$-commutative monoids is a map of the underlying monoids that commutes with all norm maps.

If $\mcC$ was actually a symmetric monoidal Mackey functor endowed with the canonical genuine $G$-symmetric monoidal structure, then the added condition on norm maps is that they should compose in the sense that
\[
N_{H}^{G}\circ N_{K}^{H} M\to M
\]
and
\[
N_{K}^{G} M\to M
\]
differ only by the natural isomorphism
\[
N_{H}^{G}\circ N_{K}^{H} M\cong N_{K}^{G} M.
\]

\begin{definition}
If $\mcC$ is a $G$-commutative monoid, then let $\GComm(\mcC)$ denote the category of $G$-commutative monoids in $\mcC$ and maps of $G$-commutative monoids.
\end{definition}

The following proposition is immediate, using the standard argument for ordinary commutative monoids.

\begin{proposition}
If $F$ is a lax $G$-monoidal functor between $G$-symmetric monoidal categories $\mcC_1$ and $\mcC_2$, then $F$ induces a functor
\[
\GComm(\mcC_1)\to\GComm(\mcC_2).
\]
\end{proposition}

While a useful language, this is too rigid to describe naturally occurring circumstances. For instance, not every symmetric monoidal subcategory of a $G$-symmetric monoidal category is a genuine $G$-symmetric monoidal category, and this is an essential feature for our localization results in Section~\ref{sec:Localization}. Similarly, categories built out of a $G$-category (like categories of modules over a ring) need not be genuine $G$-symmetric monodal categories. The categories of modules described in Section~\ref{sec:Modules} provides an example of this.

\section{$\cO$-commutative monoids}
In many naturally arising situations, we do not see all possible norm maps. We therefor want a more general notion of a $G$-symmetric monoidal category that will accommodate these situations. The essential feature of $\mSet$ that we used for the genuine $G$-commutative monoids was that it is a symmetric monoidal Green functor. In other words, it is a kind of monoid in symmetric monoidal coefficient systems, and it is closed under self-induction\footnote{This is also a kind of monodial product, where we use that $\mSet$ is the prototypical $G$-symmetric monoidal category.}.

We generalize the previous definitions using this.

\begin{definition}
An {\defemph{indexing system}} $\cO$ is a full, symmetric monoidal sub coefficient system of $\mSet$ which
\begin{enumerate}
\item contains the $G$-set $G/G$, 
\item is closed under finite limits, and
\item is closed under self-induction: if $H/K\in\cO({G/H})$ and $T \in\cO(G/K)$, then $H\times_{K}T \in\cO(G/H)$.
\end{enumerate}

An element of $\cO(G/H)$ is called an {\defemph{admissible set}} for $\cO$.
\end{definition}

\begin{example}
The constant symmetric monoidal coefficient system $\mSet^{tr}$ of sets with trivial action provides an immediate example of an indexing system. We call this one $\cO^{tr}$, the trivial coefficient system.
\end{example}

\begin{remark}
Since any indexing system $\cO$ is closed under disjoint unions and contains $G/G$, we see immediately that any indexing system contains $\cO^{tr}$.
\end{remark}

\begin{definition}
An $\cO$-symmetric monoidal category is a symmetric monoidal coefficient system $\mcC$ together with an extension of the canonical exponential map 
\[
\cO^{tr,Iso}\times\mcC\to\mcC
\]
to a bilinear map
\[
-\Box -\colon \cO^{Iso}\times\mcC\to\mcC
\]
satisfying the same associativity conditions as before.

An $\cO$-symmetric monoidal Mackey functor is a symmetric monoidal coefficient system $\mcC$ together with symmetric monoidal maps
\[
\norm_{K}^{H}\colon\mcC(G/K)\to\mcC(G/H)
\]
for every $H/K\in\cO_{G/H}$ which satisfy the double coset formula up to natural isomorphism.
\end{definition}

Of course, if $\cO=\mSet$, then we recover the earlier definition of a genuine $G$-symmetric monoidal category. By analogy with equivariant spectra indexed on a trivial universe, we call an $\cO^{tr}$-symmetric monoidal category a ``naive $G$-symmetric monoidal category''. This is no condition at all, as the only maps are given by canonical exponentiation. 

\begin{proposition}
If $\mcC$ is a symmetric monoidal coefficient system, then $\mcC$ is canonically an $\cO^{tr}$-symmetric monoidal category and an $\cO^{tr}$-Mackey functor.
\end{proposition}

\begin{remark}
It is not immediately clear from the definition that we need only define $\norm_K^H$ for admissible $H/K\in\cO(G/H)$. However, closure under subobjects and under disjoint union shows that defining the norms for the orbits gives the norms for arbitrary admissible sets. The second and third conditions on an indexing system then shows that restricting norms defined for admissibles only involves other admissibles, and similarly inducing them up.
\end{remark}

\begin{definition}
Let $\cO$ be an indexing system and let $\mcC_1$ and $\mcC_2$ be two $\cO$-symmetric monoidal categories. A functor $F\colon\mcC_1\to\mcC_2$ is a strong $\cO$-monoidal functor if for every $G/H\in\cOrb_G$ and for every $T\in\cO(G/H)$, we have a natural isomorphism
\[
T\Box F(-)\To F(T\Box -).
\]
\end{definition}


The notion of an $\cO$-symmetric monoidal structure is actually a weakening of the condition above for a genuine $G$-symmetric monoidal structure. The collection of all indexing coefficient systems forms a poset, ordered by inclusion. If $\mcC$ is a genuine $G$-symmetric monoidal category, then for any $\cO$, there is a canonical $\cO$-symmetric monoidal structure given by restriction of structure. Moreover, if $\cO'\subset\cO$, then an $\cO$-symmetric monoidal category is canonically an $\cO'$-symmetric monoidal category.

\begin{definition}
If $\cO$ and $\cO'$ are indexing coefficient systems with $\cO'\subset \cO$, and if $\mcC$ is an $\cO$-symmetric monoidal category, then an $\cO'$-commutative monoid is an object $M\in\mcC(G/G)$ and an extension
\[
\xymatrix{
{\cO'^{Iso}}\ar[r]\ar[d] & {\cO^{Iso}}\ar[r]^{-\Box M} & {\mcC.} \\
{\cO'}\ar@{-->}[urr]_{N^{(-)}(M)}}
\]
\end{definition}

In other words, we have ``norm maps'' only for the objects of $\cO'$. Since $\cO'$ is closed under subobjects, we have a collection of pairs of subgroups $K\subset H$ such that $H/K\subset \cO'$ and any $T \in\cO'$ can be expressed as a disjoint union of copies of various $H/K$. These parameterize the norms that an $\cO'$-commutative monoid has. The remaining axioms for an indexing coefficient system guarantee that
\begin{enumerate}
\item $M$ is a commutative monoid (in a traditional sense) and
\item norm maps are commutative monoid maps and compose.
\end{enumerate}

\section{Mackey Functors, Tambara Functors, and the transfer}

Much of this formalism arose from our attempt to understand the transfer. One approach in the finite group case is via the Wirtm\"{u}ller map for any $H$-spectrum $E$:
\[
G_{+}\wedge_{H}E \to F_{H}(G_{+},E)
\]
This is a natural map arising from the the fact that $G_{+}\wedge_{H} -$ is left adjoint to the forgetful functor from $G$-spectra to $H$-spectra, while $F_{H}(G_{+},-)$ is the right adjoint. The flavors of equivariant spectra reflect when this map is a weak equivalence (and for finite groups, we can reverse all of the implications): in ``naive'' spectra (traditionally those indexed on a trivial $G$-universe), this map is only an equivalence for $H=G$, while in ``genuine'' spectra (those indexed on a complete universe), this map is always an equivalence \cite{Wirth}. In between these two extremes are models of spectra (some indexed on incomplete universes, and some which are more general) for which some of these maps are equivalences for some choices of $H$.

The Wirthm\"{u}ller map being an equivalence has several amazing and mysterious consequences. For our purposes, the most important consequence is the existence of the transfer. The orbits $G/H$ become Spanier-Whitehead self-dual, and we therefore have a canonical map $G/G_{+}\to G/H_{+}$ for any subgroup $H$ dual to the canonical quotient map. Taking homotopy classes of maps out of this map produces the ``transfer'' map linking the $H$-fixed points and the $G$-fixed points for any $G$-equivariant spectrum, and morally, we should continue to think of the transfer as ``summing over the Weyl group''.

The source of the transfer has been a perpetual source of confusion, and the language of an $\cO$-commutative monoid can be used to describe where transfers arise in the corresponding infinite loop space. This is the approach taken in the study of $\Ninfty$ operads and the algebras over them in spaces and spectra \cite{BHNinfty}. Rather than describe this in full detail, we provide warm-up algebraic examples. 

We first provide a way to transition from coefficient systems (unstable data) to Mackey functors (stable data). In particular, underlying this algebraic discussion is a way to understand computationally what the maps $G/H\to G/G$ do in Mackey functor valued Bredon homology, allowing for simpler applications of equivariant cellular homology. We then discuss the ``multiplicative'' version in Mackey functors, talking about Green and Tambara functors.

\subsection{Mackey Functors}\label{sec:Coeff}
While considering equivariant homotopy theory over an incomplete universe, Lewis described intermediaries between coefficient systems and Mackey functors which contain only some of the transfers \cite[Def 1.2]{LewisMackey}. The structure and its compatibility is most easily described by an appropriate $\cO$-commutative monoid structure.

\begin{definition}
Let $\mCoeff$ denote the symmetric monoidal coefficient system assigning to $G/H$ the category of abelian group valued coefficient systems on $H$ with symmetric monoidal product given by direct sum.
\end{definition}

An object of $\mCoeff(G/H)$ is a contravariant symmetric monoidal functor
\[
\mM\colon \Set^{H}\to\Ab,
\]
and so restriction amounts to precomposition with the induction map $\Set^{H}\to\Set^{G}$:
\[
Res_{H}^{G}(\underline{M})(T )=\underline{M}(G\times_{H}T )
\]
for any $H$-set $T $. This functor has both a left adjoint and a right adjoint, denoted $Ind_{H}^{G}$ and $CoInd_{H}^{G}$ respectively.

\begin{warning}
The reader familiar with equivairant homotopy theory and with Mackey functors in general will no doubt think that the two functors $Ind_{H}^{G}$ and $CoInd_{H}^{G}$ are naturally isomorphic. While this is true for Mackey functors, this is false for coefficient systems! In particular, the usual colimit formulas for a left adjoint show that
\[
Ind_{H}^{G}(M)(G/G)=0,
\]
while
\[
CoInd_{H}^{G}(M)(G/G)=M(H/H).
\]
\end{warning}

Coinduction has a very simple formulation. Let $i_{H}^{\ast}$ denote the restriction functor from $\Set^{G}\to\Set^{H}$.

\begin{proposition}\label{prop:Coinduction}
If $\mM$ is a coefficient system on $H$, then
\[
CoInd_{H}^{G}(\mM)(T)=\mM(i_{H}^{\ast}T)
\]
for any $G$-set $T $.
\end{proposition}

This is because the functor $i_{H}^{\ast}$ is the right-adjoint to induction. Thus the induced action on coefficient systems of $i_{H}^{\ast}$ is right-adjoint to the induced map for induction, which, confusingly enough, is $Res_{H}^{G}$.

Coinduction endows $\mCoeff$ with a symmetric monoidal Mackey functor structure, and therefore a genuine $G$-symmetric monoidal category structure:
\begin{definition}\label{def:CoeffNorms}
Let $N_{H}^{G}\colon \mCoeff(G/H)\to\mCoeff(G)$ be  $CoInd_{H}^{G}$.
\end{definition}

Combining this definition with the previous proposition immediately gives the following description of the $\Box$ operation on $\mCoeff$. With this formulation, the verification of the axioms for a $G$-symmetric monoidal category follows immediately from the corresponding statements in $\mSet$.

\begin{proposition}\label{prop:CoeffBox}
The corresponding $\Box$ operation to this $G$-symmetric monoidal structure is given by
\[
T \Box\mM=\mM(T \times -).
\]
\end{proposition}

\begin{theorem}\label{thm:MackeyFunctors}
If $\mM$ is a $G$-commutative monoid in $\mCoeff$ with the coinduction $G$-symmetric monoidal structure, then $\mM$ is a Mackey functor with transfer maps given by evaluating the map
\[
\norm_{K}^{H}\colon H/K\Box i_{H}^{\ast}\mM\to i_{H}^{\ast}\mM
\]
on $H/H$.

Conversely, if $\mM$ is a Mackey functor, then $\mM$ is naturally a $G$-commutative monoid.
\end{theorem}

\begin{proof}
A unital, linear map
\[
\mM\oplus\mM\to\mM
\]
is necessarily the addition map, so the underlying commutative monoid is just $\mM$ again. The only condition we therefore need to check is the double coset formula relating the restriction of the transfer and the transfers of the restriction. We describe this only for the case of $\Tr_{H}^{G}$; the other cases are immediate by symbol replacement.

A $G$-commutative monoid is a functorial extension of $T \mapsto T \Box M$ over all of $\mSet$. Thus we have for any subgroups $K\subset H\subset G$ a commutative square:
\[
\xymatrix{
{G/H\Box \mM}\ar[rr]^{Tr_{H}^{G}}\ar[d]_{i_{K}^{\ast}\Box i_{K}^{\ast}} & & {\mM}\ar[d]^{i_{K}^{\ast}} \\
{i_{K}^{\ast} G/H \Box i_{K}^{\ast}\mM}\ar[rr]_{Res_{K}^{G} Tr_{H}^{G}} & & {i_{K}^{\ast}\mM.}
}
\]
This is the heart of the double coset formula. The restriction to $K$ to $G/H$ comes from a double coset decomposition of $G/H$:
\[
i_{K}^{\ast}G/H=\coprod_{x\in K\backslash G/H} K/K_{x},
\]
where $K_{x}$ is the stabilizer of $xH$. By construction, we then have a natural isomorphism
\[
i_{K}^{\ast} G/H \Box i_{K}^{\ast}\mM\cong \bigoplus_{x\in K\backslash G/H} K/K_{x}\Box i_{K}^{\ast}\mM.
\]
The map from this to $i_{K}^{\ast}\mM$ is the sum of all of the individual maps
\[
K/K_{x}\Box i_{K}^{\ast}\mM\to\mM,
\]
and thus we conclude the formula
\[
Res_{K}^{G} Tr_{H}^{G}=\sum_{x\in K\backslash G/H} Tr_{K_{x}}^{K},
\]
and evaluating both sides on $K/K$ gives the standard formula.

For the other implication, if $\mM$ is a Mackey functor, then since $CoInd$ for Mackey functors agrees with $CoInd$ of coefficient systems and is naturally the isomorphic to $Ind$ on Mackey functors, we have canonical maps
\[
CoInd_{H}^{G}\mM\to\mM.
\]
These provide the necessary norm maps.
\end{proof}

Thus in this additive, purely algebraic context, the ``norm'' maps are really the transfers.

\begin{remark}
If we worked instead with $\Set$-valued coefficient systems, rather than abelian group valued ones, then Theorem~\ref{thm:MackeyFunctors} only changes slightly. The $G$-commutative monoids are semigroup valued Mackey functors, rather than group valued ones.
\end{remark}

We can now provide a home for Lewis' generalization of Mackey functors. Let $U$ be a universe for $G$, and let $\cO_{D}(U)$ be the indexing coefficient system generated by those $H$-sets which $H$-embed into $U$ (the $D$ subscript indicates that this is coming from an underlying ``little discs'' operad).

\begin{theorem}[\cite{LewisMackey}]
Let $U$ be a universe, and let $\Sp^{G}(U)$ denote the category of equivariant spectra indexed on $U$. Then for any $T \in\Sp^{G}(U)$, the homotopy groups of $T $ are naturally a $\cO_{D}(U)$-commutative monoid.
\end{theorem}

We alluded to a computational advantage to this, and the advantage is a conceptual one. If $T $ is a $G$-set and $Y$ is a [genuine] $G$-spectrum, then we have a natural isomorphism
\[
T \Box\underline{\pi_{k}}(Y)\cong\underline{\pi_{k}}(T \times Y).
\]
We apply this to the cellular filtration for a genuine equivariant spectrum. This produces a long exact sequence computing Bredon homology with
\[
\underline{\pi}_\ast(T _n/T _{n-1}\wedge H\mM)=\underline{\pi}_n(T _n/T _{n-1}\wedge H\mM)=\underline{H}_n(T _n/T _{n-1};\mM)
\]
in degree $n$. The maps in the long exact sequence are induced from the maps in the triple $(T _n,T _{n-1},T _{n-2})$, and these decompose into maps of the form $G/H\to G/K$. The above discussion then shows the effect of these maps on homology: they all arise from the transfer.

\subsection{Tambara Functors}
Mackey functors have an additional symmetric monoidal operation: the $\Box$-product. This is a Day convolution product along the Cartesian product of finite $G$-sets. A unital, commutative monoid for $\Box$ is called a commutative Green functor. If $R$ is an equivariant commutative ring spectrum, then $\underline{\pi}_{0}$ has more structure than simply a commutative ring object in Mackey functors. Brun showed that $\underline{\pi}_{0}(R)$ is naturally a Tambara functor, meaning it has in addition compatible multiplicative transfer maps called norms \cite{Brun}. One of the key ingredients in Hill-Hopkins-Ravenel was a homotopically well-defined definition of a spectrum version of the norm, and we can extend this to Mackey functors.

\begin{definition}\label{defn:MackeyNorms}
Let $N_K^H\colon \Mackey_{K}\to\Mackey_{H}$ denote the composite 
\[
\mM\mapsto H\mM\mapsto N_{K}^{H}H\mM\mapsto \underline{\pi}_{0} N_{K}^{H}H\mM,
\]
where the functor on Eilenberg-MacLane spectra is the norm in spectra.
\end{definition}

\begin{proposition}\label{prop:MackeyNormsOne}
The norm functors $N_K^H$ endow the symmetric monoidal coefficient system of Mackey functors with the structure of a symmetric monoidal Mackey functor.
\end{proposition}
\begin{proof}
We need only show that the norm maps are strong symmetric monoidal functors. However, this follows immediately from the facts that
\[
\m{\pi}_0 (H\mM\wedge H\mM')\cong \mM\Box\mM',
\]
and that the norm functor on spectra is strong symmetric monoidal.
\end{proof}

By definition, the following result then holds immediately.
\begin{proposition}
The functor $\underline{\pi}_0$ is a strong $G$-monoidal map on $(-1)$-connected spectra.
\end{proposition}

The $G$-commutative monoids for this structure should immediately coincide with the Mackey functors for which the associated Eilenberg-MacLane spectrum has a commutative multiplication (really, a $G$-$E_{\infty}$ one). This is just a slight strengthening of the equivalence of the homotopy category of the heart of $G$-spectra and Mackey functors. By Brun's result, all $G$-commutative monoids are all immediately Tambara functors, and Ullman has shown by essentially the methods sketched above that the converse is also true.

\begin{theorem}[\cite{Ullman}]
If $\m{R}$ is a Tambara functor, then $H\m{R}$ is equivalent to a commutative ring spectrum.
\end{theorem}

In their thesis work, first Mazur (for cyclic $p$-groups) and then Hoyer (for all finite groups) construct a purely algebraic $G$-symmetric monoidal structure on Mackey functors extending the box product \cite{MazurArxiv}, \cite{HoyerThesis}. Hoyer then showed that these two algebraically defined constructions of norms agree with the norms in Definition~\ref{defn:MackeyNorms}. Moreover, Mazur and Hoyer gave an externalized form of Tambara functors, proving a conjecture from the first drafts of this paper. 

\begin{theorem}[{\cite{MazurArxiv}}, {\cite{HoyerThesis}}]\label{thm:Tambara}
The categories of $G$-Tambara functors and of $G$-commutative monoids in Mackey functors are equivalent.
\end{theorem}
%

\section{Localization}\label{sec:Localization}
We arrive now at one of the the motivations for this work: understanding and explaining equivariant Bousfield localization. We begin with a devastating example, originally considered by McClure in the context of Tate cohomology. Let $\bar{\rho}$ denote the quotient of the real regular representation of $G$ by its trivial subspace. The inclusion of the origin in $\bar{\rho}$ induces an equivariant map
\[
a_{\bar{\rho}}\colon S^{0}\to S^{\bar{\rho}}
\]
which is essential. However, for any proper subgroup $H\subsetneq G$, the restriction of $a_{\bar{\rho}}$ to $H$ is null (the restriction of $\bar{\rho}$ to any proper subgroup has non-trivial fixed points, and we can move the origin to the point at infinity along this subspace).

\begin{proposition}[{\cite{McClure}}]\label{prop:FundamentalExample}
The ring spectrum $S^{0}[a_{\bar{\rho}}^{-1}]$ can not be made a commutative ring spectrum.
\end{proposition}
\begin{proof}
For any proper subgroup $H$, the restriction $i_H^\ast S^0[a_{\bar{\rho}}^{-1}]$ is contractible, so we cannot have a map of commutative unital ring spectra
\[
N_{H}^{G}i_{H}^{\ast}S^{0}[a_{\bar{\rho}}^{-1}]\to S^{0}[a_{\bar{\rho}}^{-1}].
\]
\end{proof}

This is an example of nullification in the sense of Farjoun \cite{Farjoun}. The functor
\[
S^{0}\mapsto S^{0}[a_{\bar{\rho}}^{-1}]
\]
is the nullification functor associated to the localizing subcategory generated by $G/H$ for $H$ a proper subgroup and their desuspensions. Thus we conclude that localization need not preserve commutative rings!

To understand what happens, it is helpful to return to the original proof in EKMM. The following theorem is the heart of their localization results \cite[Thm VIII.2.2]{EKMM}, and it works without any change in the equivariant context. The key ingredient is an analysis of the free algebra functor on an acyclic (see \cite[Lem VIII.2.7]{EKMM} for the classical version).

\begin{theorem}\label{thm:LocalizationandOperads}
Let $L$ be a localization. If for every $L$-acyclic $Z$ the the unit map from $S^0$ to the free $\cO$-algebra $P_{\cO}[Z]$ is an $L$ equivalence, then $L$ takes $\cO$-algebras to $\cO$-algebras
\end{theorem}

For the classical case, we approximate the free commutative ring spectrum by the weakly equivalent free $E_{\infty}$-ring spectrum, and the cellular filtration shows that acyclics go to acyclics. Equivariantly, we can mirror this. However, there are many kinds of generalizations of the $E_\infty$-operads in the equivariant context, the $\Ninfty$ operads studied in \cite{BHNinfty}.

Let $U$ be a $G$-universe, and let $\mathcal L(U)$ denote the linear isometries operad built out of $U$. This is a $G$-operad. If $U$ is a trivial universe, then this is an ordinary $E_{\infty}$-operad with a trivial $G$-action. If $U$ is the complete universe, then $\mathcal L(U)$ is a $G$-$E_{\infty}$-operad, and $\mathcal L(U)$-algebras are equivalent to commutative ring spectra. For $U$ between these two extremes, $\mathcal L(U)$ is simply underlain by an ordinary $E_\infty$ operad.

Associated to $\cL(U)$ is an indexing system $\cO_{L(U)}$, defined by
\[
T \in\cO_{L(U)}({G/H})\Leftrightarrow \cL_H\big(F(T ,U),U\big)\neq\emptyset,
\]
where $\cL_H(-,-)$ is the space $H$-equivariant linear isometries (see \cite{BHNinfty} for more details).

\begin{theorem}\label{thm:Localization}
Let $L$ be a localization of the category of $G$-spectra.
\begin{enumerate}
\item If the category of $L$-acylics is an $\cO_{L(U)}$-symmetric monoidal subcategory of $\m{\Sp}$, then $L$ takes $\cL(U)$-algebras to $\cL(U)$-algebras.
\item If the category of $L$-acylics is a genuine $G$-symmetric monoidal subcategory of $\m{\Sp}$, then $L$ takes commutative rings to commutative rings.
\end{enumerate}
\end{theorem}
\begin{proof}
We need only show that if $Z$ is a cofibrant $L$-acylic, then the unit map to the free $\cL(U)$-algebra on $Z$ is an $L$-equivalence. 
Equivalently, we need only show that
\[
L\big(\cL(U)_{n+}\wedge_{\Sigma_{n}} Z^{\wedge n}\big)\simeq\ast
\]
for all $n>0$.

We show this by analyzing the skeletal filtration of $\cL(U)_n$. Let $T\in\cO_{L(U)}(G/H)$ have cardinality $n$, and let $\Gamma_T\subset G\times\Sigma_n$ denote the graph of a homomorphism $H\to\Sigma_n$ defining the $H$-set structure on $T$. This is well-defined up to conjugation. 

The equivariant he associated graded for the cellular filtration of $\cL(U)_n$ is a wedge of induced spheres of the form 
\[
(G\times\Sigma_{n}/\Gamma_{T})_+\wedge S^{m},
\] 
where $T \in\cO_{L}(U)$, and so the associated graded for $\cL(U)_{n+}\wedge_{\Sigma_{n}}Z^{\wedge n}$ is a wedge of spectra of the form
\[
\big((G\times\Sigma_{n}/\Gamma_{T})_+\wedge S^{m}\big)\wedge_{\Sigma_{n}}Z^{\wedge n}\simeq (G\times\Sigma_{n}/\Gamma_{T})_+\wedge_{\Sigma_{n}} Z^{\wedge n}\wedge S^{m},
\]
since $S^{m}$ has a trivial action. By definition of the $G$-symmetric monoidal structure on $\m{\Sp}$,
\[
(G\times\Sigma_{n}/\Gamma_{T})_+\wedge_{\Sigma_{n}} Z^{\wedge n}\simeq T \Box Z.
\]
Thus if the category of acyclics is $\cO_{L}(U)$-symmetric monoidal, then $T \Box Z$ is again acyclic, as required.

For the second part, we first note that if $U$ is a complete universe, then $\cO_{L}(U)=\mSet$. Moreover, the canonical collapse map
\[
\cL(U)_{n+}\wedge_{\Sigma_{n}} Z^{\wedge n}\to Z^{\wedge n}/\Sigma_{n}
\]
is a weak homotopy equivalence, and thus the free commutative ring spectrum on $Z$ is weakly equivalent to the free $\cL(U)$-algebra on $Z$. By our assumptions on the category of acyclics, we see that for every acyclic $Z$,
\[
Sym(Z)\simeq_{L} S^{0},
\]
and therefore we can form the localization entirely in commutative rings.
\end{proof}
\begin{corollary}\label{cor:AtLeastEinfty}
For any localization $L$ and for any commutative ring spectrum $R$, $L(R)$ is always an algebra over $\cL(\mathbb R^{\infty})$, the trivial linear isometries operad, and so necessarily has a multiplication that is unital, associative, and commutative up to coherent homotopy.
\end{corollary}
\begin{proof}
The indexing category associated to $\cL(\mathbb R^{\infty})$ is just $\cO^{tr}$.
\end{proof}

We can return to the example at the start of the section.
\begin{definition}\label{defn:FProperSG}
Let $\cO^{Prop}$ be the indexing category defined by
\[
\cO^{Prop}(G/H)=\begin{cases}
\Set & H=G, \\
\Set^{H} & \text{otherwise.}
\end{cases}
\]
\end{definition}

\begin{proposition}\label{prop:GeomFPExample}
Let $\ccZ(G/G)$ be the triangulated subcategory of $\Sp^G$ generated by all spectra of the form $G/H_{+}$ for $H$ a proper subgroup of $G$, and extend this to a symmetric monoidal coefficient system $\ccZ$ by letting $\ccZ(G/H)$ be the triangulated subcategory of $\Sp^H$ generated by the restriction of $\ccZ(G/G)$. Then $\ccZ$ is an $\cO^{Prop}$-symmetric monoidal subcategory of $\m{\Sp}$, but for no larger $\cO$ is $\ccZ$ a $\cO$-symmetric monoidal subcategory.
\end{proposition}
\begin{proof}
We first observe that for any proper subgroup $H$, $\ccZ(G/H)=\Sp^{H}$, and so the restriction to any proper subgroup is then as structured as $\Sp^{H}$ itself is. It will therefore suffice to show that for any $G/H\in\cOrb_{G}$ with $H\subsetneq G$, there is some $Z\in\ccZ(G/G)$ such that
\[
G/H\Box Z=N_{H}^{G}i_{H}^{\ast}Z\notin\ccZ(G/G).
\]
For this, it suffices to take $Z=G/H_{+}$. The restriction to $H$ of this is $S^{0}\vee Z'$ for some $H$-spectrum $Z'$, and so the norm back to $G$ is of the form $S^{0}\vee W$ for some $G$-spectrum $W$ (the exact form of which is immaterial). Since a localizing category is closed under retracts, if this were in $\ccZ(G/G)$, then $S^{0}$ would be in $\ccZ(G/G)$ and hence, $\ccZ$ would be all of $\m{\Sp}$.
\end{proof}

Thus by Theorem~\ref{thm:Localization} we see again that we should not have expected the nullification functor associated to $\ccZ$ to preserve commutative objects. White has built on this in a more general model category context \cite{White}.

\section{Module Categories}\label{sec:Modules}

While the spectra which can arise as localizations of commutative ring spectra need not be commutative ring spectra, they do necessarily have a multiplication that is commutative up to coherent homotopy. The yoga of modern structured ring spectra implies that the category of modules over such a ring spectrum is symmetric monoidal (see, for instance \cite{BHNinftyCat}). This just requires the underlying symmetric monoidal product, so we can ask now, ``Where do the norms show up in the modules over an $\cO$-commutative monoid?'' In spectra, this provides additional structure on the categories of modules. However, even in algebra is this an interesting question: what extra structure does the category of modules over a Tambara functor have?

Assume that the underlying symmetric monoidal coefficient system has levelwise colimits and that the operation $T \Box-$ commutes with sifted colimits (this is true in the most interesting cases of $G$-objects in a symmetric monoidal category in which the symmetric monoidal product commutes with colimits in each factor by \cite[Prop. A.27]{HHR}). In this case, we can make the standard definitions.

\begin{definition}\label{defn:RModules}
If $R$ is an $\cO$-commutative monoid, then an $R$-module is an object $M$ of $\mcC(G/G)$ together with a map
\[
\mu\colon R\otimes M\to M
\]
such that the following diagrams commute
\[
\xymatrix{
{R\otimes R\otimes M}\ar[r]^{m\otimes 1} \ar[d]_{1\otimes\mu} & {R\otimes M}\ar[d]^{\mu} \\
{R\otimes M}\ar[r]_{\mu} & {M}
}\quad \text{and}\quad
\xymatrix{
{{\mathbb S}\otimes M}\ar[r]^{\eta\otimes 1}\ar[dr]_{\cong} & {R\otimes M}\ar[d]^{\mu} \\
& {M,}}
\]
where $\mathbb S$ is the symmetric monoidal unit, $\eta$ is the unit map for $R$, and the map labeled $\cong$ is the canonical isomorphism.
\end{definition}

Since the underlying monoid for an $\cO$-commutative monoid is a just commutative monoid, we can form the usual ``balanced tensor product'' construction to build a symmetric monoidal product on the category of $R$-modules.

\begin{definition}\label{def:TensorOverR}
If $M$ and $N$ are $R$-modules for an $\cO$-commutative monoid $R$, then define $M\otimes_{R}N$ to be the coequalizer of
\[
\xymatrix{
{M\otimes R\otimes N}\ar@<1ex>[rr]^{1_{M}\otimes\mu_{N}}\ar@<-1ex>[rr]_{\mu_{M}\otimes 1_{N}} & & {M\otimes N.}
}
\]
\end{definition}

This isn't quite a loose enough definition. In particular, $\mathbb S\mhyphen\Mod$ should be $\mcC$. To make the category of $R$-modules into a symmetric monoidal Mackey functor, we loosen the requirements on where $M$ lives.

\begin{definition}\label{defn:RMod}
If $R$ is an $\cO$-commutative ring, then let $\mRMod$ denote the symmetric monoidal coefficient system whose value at $G/H$ is $i_{H}^{\ast}\RMod$.
\end{definition}

Even with this looser definition, when we speak of an $R$-module, unless otherwise specified, we assume that it is an object in $\mcC(G/G)$ together with the appropriate structure maps. If it lives in $\mcC(G/H)$, then we replace all instances of $G$ in discussions that follow with $H$.

The fact that $\mcC$ was a symmetric monoidal coefficient system guarantees that $\mRMod$ has the appropriate restriction maps. We want now to show that in fact, it is an $\cO$-Mackey functor.

\begin{proposition}\label{prop:T BoxRMod}
Let $T $ be a $G$-set. Then for any $R$-module $M$, $T \Box M$ is naturally (in $T $ and $M$) a $T \Box R$-module.

If $M$ and $N$ are $R$-modules, then
\[
T \Box (M\otimes_R N)\cong (T \Box M)\otimes_{T \Box R} (T \Box N).
\]
\end{proposition}
\begin{proof}
Since the operation $T \Box (-)$ is symmetric monoidal and functorial, applying it to $R\otimes M\to M$ produces
\[
(T \Box R)\otimes (T \Box M)\cong T \Box (R\otimes M)\to T \Box M.
\]
The commutativity of the required diagrams is standard.

For the second part, we use the observation that a symmetric monoidal functor which commutes with sifted colimits preserves reflexive coequalizers. Since the coequalizer describing $M\otimes_{R} N$ is a reflexive one (using the unit), we conclude that for any $R$-modules $M$ and $N$,
\[
T \Box (M\otimes_{R} N)\cong (T \Box M)\otimes_{T \Box R} (T \Box N)
\]
as $T \Box R$-modules.
\end{proof}

It is here that we use the $\cO$-commutative monoid structure on $R$. We have a natural map of commutative monoids $T \Box R\to R$ for any $T \in\cO$, using the canonical map $T \to\ast$.

\begin{definition}\label{defn:RModBox}
If $R$ is a $\cO$-commutative monoid, $M$ is an $R$-module, and $T \in\cO_{G/G}$, then let
\[
T \Box_{R}M=R\otimes_{T \Box R}(T \Box M).
\]
\end{definition}

\begin{theorem}\label{thm:RModFSymMonoidal}
If $R$ is an $\cO$-commutative monoid, then the operation $\Box_{R}$ makes $\mRMod$ into a $\cO$-symmetric monoidal coefficient system.
\end{theorem}

\begin{proof}
The naturality of $T \Box_{R}M$ in the first factor for isomorphisms follows from the fact that $\ast$ is the terminal object in $\cO$ and therefore any triangle with two legs the canonical maps to $\ast$ commutes. The naturality of $T \Box_{R}M$ in the second factor follows from the functoriality of $M\mapsto T \Box M$ in $R$-modules.

It remains to show that this is symmetric monoidal in both factors. For the first factor, we have natural isomorphism
\begin{align*}
(T \amalg Y)\Box_{R} M & =  R\otimes_{(T \amalg Y)\Box R}(T \amalg Y)\Box M \\
& \cong  R\otimes_{(T \Box R)\otimes (Y\Box R)} (T \Box M)\otimes (Y\Box M) \\
& \cong  (R\otimes_{R} R)\otimes_{(T \Box R)\otimes (Y\Box R)} (T \Box M)\otimes (Y\Box M) \\
& \cong  (R\otimes_{T \Box R} T \Box M)\otimes_{R} (R\otimes_{Y\Box R} Y\Box M) \\
& =  (T \Box_{R} M)\otimes_{R}(Y\Box_{R}M),
\end{align*}
where the first isomorphism is from the structure of $\Box$ and the remaining ones are standard.

The symmetric monoidalness in the second factor requires only slightly more work:
\begin{align*}
T \Box_{R} (M\otimes_{R} N) & =  R\otimes_{T \Box R} T \Box (M\otimes_{R} N) \\
& \cong  R\otimes_{T \Box R} (T \Box M)\otimes_{T \Box R} (T \Box N)\\ 
& \cong  R\otimes_{R} R\otimes_{T \Box R} (T \Box M)\otimes_{T \Box R} (T \Box N) \\
& \cong  \big(R\otimes_{T \Box R} (T \Box M)\big)\otimes_{R}\big(R\otimes_{T \Box R} (T \Box N)\big) \\
& =  (T \Box_{R}M)\otimes_{R}(T \Box_{R}N).
\end{align*}
The remaining properties are inherited from $\mcC$.
\end{proof}

We can also link free modules and the $\cO$-symmetric monoidal structure.

\begin{definition}\label{defn:Induced}
If $Z\in\mcC(G/G)$, then let $R\otimes Z$ be the obvious $R$-module.
\end{definition}

\begin{proposition}\label{prop:BoxRofFrees}
For any $Z\in\mcC$ and for any $T \in\cO_{G/G}$, we have a natural isomorphism of $R$-modules
\[
T \Box_{R}(R\otimes Z)\cong R\otimes (T \Box Z).
\]
\end{proposition}
\begin{proof}
The $T \Box R$-module $T \Box (R\otimes Z)$ is just $(T \Box R)\otimes (T \Box Z)$. The formula then follows immediately by the standard base-change isomorphisms.
\end{proof}

This provides intuition for the (albeit homotopical) localization result. Let $R$ be a commutative ring. If $L(R)$ is an $\cO$-commutative monoid, then the category of $L(R)$-modules is $\cO$-symmetric monoidal. Consider an acyclic $Z$. Then $L(R)\otimes Z$ is contractible. However, by the previous proposition, for any $T \in\cO(G/G)$, we have
\[
T \Box_{L(R)}(L(R)\otimes Z)\cong L(R)\otimes (T \Box Z).
\]
It must therefore be the case that $L(R)\otimes (T \Box Z)$ is also contractible.

If $\mcC$ is an $\cO$-symmetric monoidal Mackey functor, then the category of $R$-modules is naturally so too.

\begin{definition}\label{def:RelativeNorms}
Let $R$ be an $\cO$-commutative monoid in an $\cO$-symmetric monoidal Mackey functor $\mcC$. Then define functors $\normR_{K}^{H}\colon \mRMod(G/K)\to \mRMod(G/H)$, by
\[
\normR_{K}^{H}(M)= i_{H}^{\ast}R\otimes_{N_{K}^{H}i_{K}^{\ast}R}N_{K}^{H}M.
\]
\end{definition}

\begin{theorem}\label{thm:RModNorms}
With norms defined by $\normR_{K}^{H}$, $\mRMod$ becomes an $\cO$-symmetric monoidal Mackey functor.
\end{theorem}
The proof of this follows from the corresponding statement about the $\cO$-symmetric monoidal category structure, together with the double-coset formulae in $\mcC$.

For any universe $U$ and for any algebra $R$ over $\cL(U)$, there is a good category of modules for $R$ which is an $\cO_L(U)$-symmetric monoidal Mackey functor \cite{BHNinftyCat}. This generalizes the classical result for commutative algebras, and it is expected to hold for any $\Ninfty$ operad instead of just linear isometries operads. Using this, we can show that Bousfield localization relative to an $\cL(U)$-algebra $R$ always preserves $\cL(U)$-algebras.

\begin{theorem}
If $\m{R}$ is an $\cL(U)$-algebra or a commutative algebra, then the Bousfield localization $\L_R(-)$ preserves $\cO$-algebras for any $\Ninfty$ operad $\cO$ mapping to $\cL(U)$.
\end{theorem}
\begin{proof}
By Theorem~\ref{thm:Localization}, it suffices to show that for any (cofibrant) acyclic $Z$, the spectrum
\[
N_K^Hi_K^\ast Z
\]
is acyclic for any $G/H\in\cOrb_G$ and for any admissible $H/K\in\cO(G/H)$. Since $\cO$ is an $\Ninfty$ operad mapping to $\cL(U)$, the admissibles for $\cO$ are all admissible for $\cL(U)$ \cite{BHNinfty}, and hence it suffices to consider the case that $\cO=\cL(U)$. In this case, the result follows immediately from applying $\normR_K^H$ to the weak equivalence
\[
i_K^\ast(R\wedge Z)\xrightarrow{\simeq} \ast
\]
and expanding out the left-hand side.
\end{proof}

\section{Other Categories}

\subsection{Compact Lie Groups}
Our discussion has focused on finite groups, but largely because all subgroups of a finite group are of finite index. This means that all homogeneous spaces for finite groups are also finite, and thus classified by a map from $G$ (or a subgroup) to some symmetric group. This arises in a subtle way: the map $G/H\to G/G$ is a finite covering map for any $H\subset G$, and hence we have norms from $H$ to $G$.

For a compact Lie group of positive dimension, this is no longer the case. In fact, very curious things can happen. We have a basic observation, though.

\begin{proposition}\label{prop:CompactLieNorms}
If $K\subset H$ is a closed subgroup of finite index, then we have a symmetric monoidal norm functor
\[
N_{K}^{H}\colon \Sp^{K}\to\Sp^{H}.
\]
The same is true for $K$-objects in any [topological] symmetric monoidal category.
\end{proposition}

We can now play the same game we played for finite groups. The difficulty, however, is that the role of $\mSet$ is still played by the symmetric monoidal coefficient system that assigns to $G/H$ the category of all {\emph{finite}} $H$-sets. In particular, if $H$ is connected, then $\Set^{H}$ is just $\Set$.

\begin{remark}
The reader might wonder why we consider only finite index closed subgroups. The heart of the matter is the nature of a $G$-$E_{\infty}$ operad and genuine $G$-spectra. These operads have the distinguished feature that if $\cO$ is a $G$-$E_{\infty}$-operad (satisfying certain cofibrancy conditions), then the canonical map
\[
\cO_{n+}\times_{\Sigma_{n}} E^{\wedge n}\to E^{\wedge n}/\Sigma_{n}
\]
is a $G$ weak equivalence for a cofibrant genuine $G$-spectrum $E$, and thus understanding the symmetric powers of a $G$-spectrum is tantamount to understanding these homotopy symmetric powers.

By definition, the $n$\textsuperscript{th} space in a $G$-$E_{\infty}$-operad is a universal space for the family of closed subgroups of $G\times\Sigma_{n}$ which intersect $\Sigma_{n}$ trivially, and a closed subgroup intersects $\Sigma_{n}$ trivially if and only if is the graph of a homomorphism from a closed subgroup $H\subset G$ to $\Sigma_{n}$. Thus we reduce to the case of understanding closed subgroups of $G$ together with a finite index subgroup thereof.
\end{remark}

Theorem~\ref{thm:Localization} still holds exactly as written, with no modifications to the proof needed. 
\begin{theorem}
Let $G$ be a compact Lie group and let $L$ be a localization of the category of genuine $G$-spectra.
\begin{enumerate}
\item If the category of $L$-acyclics is an $\cO_L(U)$-symmetric monoidal subcategory of $\m{\Sp}$, then $L$ takes $\cL(U)$-algebras to $\cL(U)$-algebras.
\item If the category of $L$-acyclics is a genuine $G$-symmetric monoidal subcategory of $\m{\Sp}$, then $L$ takes commutative ring spectra to commutative ring spectra.
\end{enumerate}
\end{theorem}

However, our intuition is very different! An example showing how weird this can be and how weirdness is in some sense generic is in order. Consider $G=S^{1}$, where essentially none of our intuition from finite groups holds.

\begin{example}
Let $G=S^{1}$, and let $R$ be a commutative $S^{1}$-spectrum. Then since $S^{1}$ has no finite index subgroups, from the point of view of $S^{1}$ itself, $R$ has a multiplication up to homotopy that is coherently unital, associative, and commutative. However, any closed proper subgroup of $S^{1}$ is finite and therefore has loads of finite index subgroups. Thus the restriction of $R$ to any proper subgroup $H$ results in a commutative $H$-ring spectrum (the kind described above). In the modern parlance, a commutative $S^{1}$-ring spectrum looks like a naive commutative ring spectrum, but it restricts to a genuine commutative ring spectrum for any proper subgroup.
\end{example}

As a consequence of this, the fundamental localization counterexample, Proposition~\ref{prop:FundamentalExample}, does not work the same way in the compact Lie case.

\begin{proposition}\label{prop:CompactLieCounterexample}
Let $\ccZ$ be the triangulated subcategory generated by all spectra of the form $S^1/H_+$ for $H\subsetneq S^1$, and let $L(-)$ be the associated nullification functor. Then $L(-)$ preserves commutative ring spectra.
\end{proposition}

Put another way, let $V_{n}$ be a sequence of representations of $S^{1}$ with trivial fixed points and such that every irreducible representation of $S^{1}$ (besides the trivial one) occurs in some $V_{n}$, and let $a_{V_{n}}\colon S^{0}\to S^{V_{n}}$ be the inclusion of the origin and the point at infinity. Then $a_{V_{n}}$ is essential (though here we do not know that the restriction to any particular proper subgroup is trivial), but for any $H\subsetneq S^1$, there is an $n$ such that
\[
i_{H}^{\ast}(a_{V_1}\dots a_{V_n})=0.
\]
Then the ring
\[
S^{0}[a_{V_{1}}^{-1},a_{V_{2}}^{-1},a_{V_{3}}^{-1},\dots]
\]
is a commutative $S^{1}$ ring spectrum.

In fact, the localization in Proposition~\ref{prop:CompactLieCounterexample} is a smashing localization, and the spectrum with which we smash, $L(S^0)$, is the suspension spectrum of $\tilde{E}\mathcal P$, where $\mathcal P$ is the family of proper subgroups. 

\subsection{And Beyond}
All of what we have described relied on the underlying presheaf of symmetric monoidal categories on the orbit category. This provided a description of equivariant spectra and coefficient systems (which in turn were kinds of pre-sheaves with transfers) which tied together the multiplications and presheaf structures. This type of argument should then hold generically true for presheaves on a topos. Before continuing, we present a warning, showing that our localization result will also hold motivically, provided the realization functor from motivic spectra over $\mathbb R$ to $C_2$-equivariant spectra is lax monoidal.

\begin{theorem}
Consider the motivic stable category over $\mathbb R$. Let $\rho\colon S^{0}=\mu_{2}\to \mathbb G_{m}$ be the inclusion of $\pm 1$ into $\mathbb G_{m}$. Then the localization $S^{0}[\rho^{-1}]$ cannot be a commutative ring spectrum.
\end{theorem}
\begin{proof}
The simplest proof is to take complex points. This produces a symmetric monoidal functor from the motivic stable category over $\mathbb R$ to genuine $\mathbb Z/2$-equivariant spectra, and the value of this on $S^{0}[\rho^{-1}]$ is $S^{0}[a_{\bar{\rho}}^{-1}]$.
\end{proof}

This shows that we need a similar language to understand these enriched commutative rings in presheaves on a topos. For motivic homotopy, the analogous indexing category is built from schemes over a fixed scheme. In this framework, the symmetric monoidal coefficient system $\mSet$ should be replaced by a [free] symmetric monoidal category on the comma category of maps to an object. Many of our results should then hold with little change.

\begin{remark}
The motivic context over a fixed field $k$ looks like it will blend the finite and compact Lie cases. Associated to a variety $V$, Hu describes a tensoring operation on commutative ring spectra analogous to our norms \cite{Hu}. If $V=Spec(F)$ for a finite extension field $F$ of $k$, then this functor should extend to a symmetric monoidal functor on motivic spectra. Thus we see norms for any finite extension fields, but on commutative rings, we also see additional norm maps.
\end{remark}

\bibliographystyle{plain}

\bibliography{G-Symm}

\end{document}